\documentclass[unsortedaddress,onecolumn, nofootinbib,10pt]{revtex4}
\usepackage[a4paper,bindingoffset=0.2in, left=0.5in,right=0.8in,top=1in,bottom=1in, footskip=.25in]{geometry}
\usepackage[utf8]{inputenc}
\usepackage{amsmath,empheq,color}
\usepackage{amsfonts}
\usepackage{amsthm}
\usepackage{amssymb}
\usepackage{graphicx}
\usepackage{hyperref}
\usepackage[normalem]{ulem}
\usepackage{epigraph}
\usepackage{mathrsfs}
\usepackage{bbm}
\usepackage{wrapfig}
\usepackage{lineno}
\usepackage{scalerel,stackengine}
\usepackage[T1]{fontenc}
\usepackage{epigraph}
\usepackage{cancel}
\usepackage{soul}
\usepackage{ulem}
\usepackage{hyperref}
\usepackage{amsmath,amssymb,amsfonts,amsthm,amscd}
\usepackage{graphicx}
\usepackage{enumerate}
\usepackage{colordvi}
\usepackage{units}
\usepackage{epsfig}
\usepackage{natbib}
\usepackage{enumerate}
\usepackage{colordvi}
\usepackage{multirow}
\usepackage{afterpage}
\usepackage[utf8]{inputenc}
\usepackage{xcolor}
\usepackage{tikz-cd}
\usepackage{subfigure}
\usepackage{cancel}

\usepackage[textwidth=1.7cm,textsize=small]{todonotes}

\newtheorem{theorem}{Theorem}
\newtheorem{remark}{Remark}
\newtheorem{corollary}{Corollary}
\newtheorem{example}{Example}
\newtheorem{definition}{Definition}
\newtheorem{proposition}{Proposition}
\newtheorem{lemma}{Lemma}

\def\be{\begin{equation}}
\def\ee{\end{equation}}
\def\ba{\begin{eqnarray}}
\def\ea{\end{eqnarray}}

\newcommand{\cL}{\mathcal{L}}
\newcommand{\cO}{\mathcal{O}}

\newcommand{\cH}{\mathcal{H}}
\newcommand{\X}{\mathcal{X}}

\DeclareMathOperator*{\argmin}{argmin}

\newcommand{\R}{\mathbb{R}}
\newcommand{\e}{\mathrm{e}}
\newcommand{\ttt}{\tiny\text}

\newcommand{\norm}[1]{\left\lVert#1\right\rVert}

\usepackage{todonotes}

\begin{document}

\title{Bregman dynamics, contact transformations and convex optimization}

\author{Alessandro Bravetti}
\email{alessandro.bravetti@iimas.unam.mx} 
\affiliation{Instituto de Investigaciones en Matem\'aticas Aplicadas y en Sistemas, 
Universidad Nacional Aut\'onoma de M\'exico, A.~P.~70543, M\'exico, DF 04510, Mexico}
\author{Maria L. Daza-Torres}
\email{mdazatorres@ucdavis.edu} 
\affiliation{Department of Public Health Sciences, 
University of California Davis, California 95616, United States}

\author{Hugo Flores-Arguedas}
\email{hfloresarguedas@astate.edu}
\affiliation{Arkansas State University, Campus Quer\'etaro (ASUCQ)}

\author{Michael Betancourt}
\email{betan@symplectomorphic.com}
\affiliation{Symplectomorphic LLC, New York, USA}

\begin{abstract}
Recent research on accelerated gradient methods of use in optimization 
has demonstrated that these methods can 
be derived as discretizations of dynamical systems. 
This, in turn, has provided 
a basis for more systematic investigations, especially into the geometric structure 
of those dynamical systems and their structure--preserving discretizations.   
In this work, we introduce dynamical systems defined through a \emph{contact geometry} which 
are not only naturally suited to the optimization goal but also subsume all 
previous methods based on geometric dynamical systems.  
As a consequence, all the deterministic flows used in optimization share an extremely 
interesting geometric property: they are invariant under contact transformations.
In our main result, we exploit this observation to show that the celebrated Bregman 
Hamiltonian system can always be transformed
into an equivalent but separable Hamiltonian by means of a contact transformation.
This in turn enables the development of fast and robust discretizations through geometric \emph{contact 
splitting integrators}.  
As an illustration, we propose the Relativistic Bregman algorithm, and show in  
some paradigmatic examples that it compares favorably with respect to standard optimization
algorithms such as classical momentum and Nesterov's accelerated gradient.
\end{abstract}

\maketitle


\section{Introduction}

Despite their practical utility and explicit convergence bounds, accelerated gradient 
methods have long been difficult to motivate from a fundamental theory.  This lack of 
understanding limits {the theoretical foundations of the methods, which in turn hinders  
the development of new and more principled schemes}.  
Recently a progression
of work has studied the continuum limit of accelerated gradient methods, demonstrating 
that these methods can be derived as discretizations of continuous dynamical systems. 
Shifting the focus to the structure and discretization of these latent dynamical systems 
provides a foundation for the systematic development and implementation of new 
accelerated gradient methods.

This recent direction of research began with~\cite{su2014differential}, which found a 
continuum limit of \emph{Nesterov's accelerated gradient} method (NAG)
%
\begin{align}
X_{k}&=P_{k-1}-\tau\nabla f(P_{k-1})\label{eq:NAG1}
\\
P_{k}&=X_k+\frac{k-1}{k+2}(X_k-X_{k-1})\,,\label{eq:NAG2}
\end{align}
%
by discretizing the ordinary differential equation
\begin{equation}\label{eq:DiffEqNAG}
	\ddot X+\frac{3}{t}\dot X+\nabla f(X)=0\,,
\end{equation}
for $t>0$ with the initial conditions $X(0)=X_0$ and $\dot X(0)=0$.  By generalizing
the ordinary differential equation~\eqref{eq:DiffEqNAG} they were then able to  
derive similar accelerated gradient methods that achieved comparable convergence rates.

Later, in~\cite{wibisono2016variational} the authors found what is arguably still the most important development 
in this direction, by showing that accelerated methods can
also be derived as discretizations of a more structured family of \emph{variational} dynamical 
systems, specified with a time--dependent Lagrangian function, or equivalent Hamiltonian function.
Consider an objective function $f: \X\rightarrow \mathbb R$, which is continuously 
differentiable, convex, and has a unique minimizer $X^*\in \X$.  Moreover 
assume that $\X$ is a convex set endowed with a distance--generating function 
$h:\X\rightarrow \R$ that is also convex and essentially smooth.  From the 
\emph{Bregman divergence} induced by $h$,
\begin{equation}\label{eq:BD}
	D_h(Y,X)=h(Y)-h(X)-\langle\nabla h(X),Y-X\rangle
\end{equation}
they derived the \emph{Bregman Lagrangian}
\begin{equation}\label{eq:BL}
	L_{\ttt{Br}}(X,V,t)=\e^{\alpha+\gamma}\left(D_h(X+\e^{-\alpha}V,X)-\e^{\beta}f(X)\right)\,,
\end{equation}
where $\alpha$, $\beta$, and $\gamma$ are continuously differentiable functions 
of time.  
They then proved that under the \emph{ideal scaling conditions} 
\begin{align}
	\dot{\beta}&\leq \e^{\alpha}\label{eq:IS1}\\
	\dot \gamma&=\e^{\alpha}\label{eq:IS2}\,,
\end{align}
the solutions of the resulting Euler--Lagrange equations
\begin{equation}\label{eq:BEL}
	\ddot X+\left(\e^{\alpha}-\dot\alpha\right)\dot X+\e^{2\alpha+\beta}\left[\nabla^2h(X+\e^{-\alpha}\dot X)\right]^{-1}\nabla f(X)=0
\end{equation}
satisfy~\cite[Theorem 1.1]{wibisono2016variational}
\begin{equation}\label{eq:ExpConv}
	f(X)-f(X^*)\leq \cO(\e^{-\beta})\,.
\end{equation}

From a {physical} perspective the two terms in equation~\eqref{eq:BL} play the 
role of a kinetic and a potential energies, respectively.  At the same time the explicit 
time--dependence of the Lagrangian~\eqref{eq:BL} is a necessary ingredient in order 
for the dynamical system to dissipate energy and relax to a minimum of the potential,
and hence to a minimum of the objective function.  Moreover, by~\eqref{eq:IS1}, the 
optimal convergence rate is achieved by choosing $\dot\beta=\e^{\alpha}$, 
i.e.~$\beta=\int_{t_0}^t\e^{\alpha(s)}ds$, and we observe that in the Euclidean case, 
$h(X)=\frac{1}{2}\norm{X}^2$, the Hessian is the identity matrix and 
thus~\eqref{eq:BEL} simplifies to
\begin{equation}\label{eq:BELEuclidean}
	\ddot X+\left(\e^{\alpha}-\dot\alpha\right)\dot X+\e^{2\alpha+\beta}\nabla f(X)=0\,.
\end{equation}
Finally the authors developed a heuristic discretization of~\eqref{eq:BEL} that yielded 
optimization algorithms matching the continuous convergence rates.
However, these algorithms are of no practical use, due to the extremely high cost of their implementation.

In~\cite{betancourt2018symplectic} the authors considered more systematic discretizations of these
variational dynamical systems that exploited the fact that they are well suited for 
numerical discretizations that preserve their geometric 
structure~\cite{hairer2006geometric}.  
%
	%
%
%
In particular they considered the 
associated \emph{Bregman Hamiltonian}~\cite{wibisono2016variational},
\begin{equation}\label{eq:BH}
	H_{\ttt{Br}}(X,P,t)=\e^{\alpha+\gamma}\left(D_{h^*}(\e^{-\gamma}P+\nabla h(X),\nabla h(X))+\e^{\beta}f(X)\right)\,,
	\end{equation}
where $h^*(P):=\sup_V\langle P, V\rangle-h(X)$ is the Legendre transform of $h(X)$, and then
argued that a principled way to obtain reliable and rate--matching discretizations 
of the resulting dynamical system
\begin{align}
	\dot X&=\nabla_PH_{\ttt{Br}}(X,P,t)\label{eq:H1}\\
	\dot P&=-\nabla_XH_{\ttt{Br}}(X,P,t)\label{eq:H2}\,,
\end{align}
is with \emph{pre--symplectic integrators} in an extended phase space where $t$ and its associated momentum $E$ are
added as dynamical variables.
They numerically demonstrated in the Euclidean (i.e.,~separable) case that a standard leapfrog integrator 
yields an optimization algorithm that achieves polynomial convergence rates and showed how the introduction 
of a gradient flow could achieve late--time exponential convergence rates matching those seen empirically in 
other accelerated gradient methods.

A more theoretical approach to rate--matching geometric discretizations for the Bregman dynamics has been then proposed in~\cite{francca2020dissipative}, 
where the authors prove that pre--symplectic integrators provide a principled way to discretize Bregman dynamics while preserving the continuous rates of convergence 
up to a negligible error, provided some assumptions are met.
The crucial significance of this result lies in the fact that it implies that splitting integrators automatically yield ``rate--matching'' algorithms, without the need for a discrete convergence analysis.
Ideally, this one would like to be able to directly apply splitting integrators to the general Bregman Hamiltonian. Unfortunately, however, when applied to the general (non--separable) Bregman dynamics, these splitting 
methods yield implicit updates that increase the computational burden and affect the numerical stability.
As a first attempt to remedy this problem, in~\cite{francca2020dissipative} the authors used a strategy that involves doubling the phase space dimension. As we will comment in the following, this strategy is not completely satisfying (see also the discussion in~\cite{duruisseaux2021adaptive}). 
Other attempts have been proposed later in~\cite{duruisseaux2021adaptive,campos2021discrete}, based respectively on the so--called \emph{Poincar\'e transform} and the \emph{adaptive Hamiltonian variational integrators} in the first case and on the \emph{cosymplectic geometry} and the associated discrete variational methods for time--dependent Lagrangians in the second. Although both perspectives seem to provide a robust solution to the problem, they are rather sophisticated compared to the simplicity of the splitting mechanism, and in particular, it is not clear how to adapt the fundamental results from~\cite{francca2020dissipative} to this setting.

Therefore, despite much investigation, there is still a crucial question about the Bregman dynamics
that is left open: 
\begin{center}
    is it possible to find an explicit splitting integrator for the general (non--separable)
Bregman Hamiltonian?
\end{center}
Note that, from all the above discussion, it would seem that the answer to this question is negative.
However, as we will see, a proper geometric approach will reveal otherwise. To understand this result and to provide the complete general picture, it is convenient to first conclude our brief account of the geometric approaches to the construction of dynamical systems that can be used in optimization: it turns out that we can replace variational dynamical systems that exploit heuristic time--dependencies to achieve dissipation with explicitly \emph{dissipative} dynamical 
systems.  \cite{muehlebach2019dynamical} and~\cite{diakonikolas2021generalized} considered
a dynamical systems perspective on these systems, showing how relatively simple dissipations can achieve state-of-the-art convergence.  \cite{francca2019conformal} 
took a more geometric perspective, replacing the time--dependent Hamiltonian geometry of
the variational systems with a \emph{conformally symplectic} geometry that generates
dynamical systems of the form
\begin{align}
	\dot X&=\nabla_PH(X,P)\label{eq:CS1}\\
	\dot P&=-\nabla_XH(X,P)-c\,P\label{eq:CS2}\,,
\end{align}
with $c\in\R$ a constant.  Being a geometric dynamical system this approach also 
admits effective geometric integrators similar to \cite{betancourt2018symplectic}.
These conformally symplectic dynamical systems, however, are less general than
the time--dependent variational dynamical systems; in particular NAG cannot be 
exactly recovered in this framework~\cite{francca2019conformal}.

Another relevant aspect that has been uncovered by studying optimization algorithms from
a variational or Hamiltonian analysis is the focus on a very important degree
of freedom, the choice of the kinetic energy, that plays a fundamental role in the construction
of fast and stable algorithms that can possibly escape local minima, in direct analogy with what 
happens~in Hamiltonian 
Monte Carlo methods~\cite{betancourt2017geometric,betancourt2017conceptual,livingstone2019kinetic}. 
In particular, first~\cite{maddison2018hamiltonian} and 
then~\cite{francca2019conformal} have
motivated
 that a careful choice of the kinetic energy term can stabilize
the dynamical systems when the objective function is rapidly changing, similar to 
the regularization induced by trust regions. {Indeed,} like the popular Adam, AdaGrad and RMSprop
algorithms, the resulting \emph{Relativistic Gradient Descent} (RGD) algorithm  
regularizes the dynamical velocities to achieve a faster convergence and improved 
stability.

{Finally~\cite{maddison2018hamiltonian}} introduced another way of incorporating dissipative terms 
into Hamilton's equations~\eqref{eq:H1}--\eqref{eq:H2}. 
Their \emph{Hamiltonian descent}
algorithm is derived from the equations of motion 
\begin{align}
	\dot X&=\nabla_PH(X,P)+X^{*}-X\label{eq:HD1}\\
	\dot P&=-\nabla_XH(X,P)+P^{*}-P\label{eq:HD2}\,,
\end{align}
where $(X^{*},P^{*})=\argmin H(X,P)$.   Because the dynamics are defined using terms only 
linear in $X$ and $P$ they converge to the optimal solution exponentially quickly under 
mild conditions on $H$~\cite{o2019hamiltonian}.  
That said, this exponential convergence 
requires already knowing the optimum $(X^{*},P^{*})$ in order to generate the dynamics.  
Additionally this particular dynamical system lies outside of the variational and conformal 
symplectic families of dynamical systems and so can not take advantage of the geometric 
integrators.

In this work we show that all of the {above--mentioned dynamical systems} 
can be incorporated into a 
single family of \emph{contact Hamiltonian systems}~\cite{Bravetti2017a,bravetti2017contact}
endowed with \emph{a contact geometry}.  The geometric foundation provides a powerful 
set of tools for both studying the convergence of the continuous dynamics as well as
generating structure--preserving discretizations.  
More importantly, the geometric character of these dynamics directly implies that they are 
invariant under \emph{contact transformations}. This indeed will be the key property to prove
that the Bregman 
dynamics can always be re--written in new coordinates where the Hamiltonian is \emph{separable}, 
thus establishing a definitive positive answer to our guiding question (see Theorem~\ref{th:Kbr}). 
This is the main result of our work.
We argue that this property is of fundamental interest for practitioners, as it enables 
to directly construct simple explicit structure--preserving discretizations in the original phase space of the system.
Finally, equipped with all these results,
we introduce the \emph{Relativistic Bregman Dynamics},
and provide an explicit contact integrator that accurately follows the 
continuous flow in
the original phase space of the system, thus obtaining the associated \emph{Relativistic Bregman algorithm}. 
We also include numerical experiments showing that this simple construction
is comparable to state-of-the-art approaches to discretize the (non--separable) 
Relativistic Bregman dynamics in terms of stability and rates of convergence (see~\cite{Daza-Torres2022_GitHub}).

The structure of this work is as follows: in Section~\ref{sec:CHS} we introduce contact 
Hamiltonian systems and show that all systems corresponding to 
Equations~\eqref{eq:BELEuclidean}, \eqref{eq:CS1}--\eqref{eq:CS2}, and~\eqref{eq:HD1}--\eqref{eq:HD2}
can be easily recovered as particular examples. 
Based on this result, we stress the importance of
(time--dependent) contact transformations 
and then we prove that they can be used to make the Bregman Hamiltonian separable.
To provide an explicit  
example of this construction, 
we then introduce the Relativistic Bregman Dynamics. 
Indeed, after applying Theorem~\ref{th:Kbr}, we obtain an equivalent but separable
Hamiltonian system that can be integrated by splitting,
thus obtaining what we call the related Relativistic Bregman algorithm (RB).
Then in Section~\ref{sec:numerics} 
we illustrate numerically that the RB can 
perform as well as other state-of-the-art algorithms in standard optimization tasks 
(see also~\cite{Daza-Torres2022_GitHub} for further numerical tests and for all the codes).
Finally, in Section~\ref{sec:Conclusions} we summarize the results and discuss
future directions.

\section{Contact geometry in optimization}\label{sec:CHS}

\subsection{Definitions and examples}\label{subsec:standard}
Contact geometry is a rich subject at the intersection among differential geometry, topology and dynamical systems.
Here in order to ease the exposition, we will present some of the basic facts needed to compare with previous works using symplectic
and pre--symplectic structures in optimization in 
full generality, but we will soon specialize them
to the cases of interest.
For a treatment of the more general theory
see~\cite{Arnold,geiges2008introduction,bravetti2018contact,bravetti2017contact,ciaglia2018contact,deleon2017cos,gaset2019new,de2019contact}.

	\begin{definition}\label{def:contactM}\rm
	A \emph{contact manifold} is a pair $(C,\mathcal D)$, where $C$ is a ($2n+1$)--dimensional manifold and $\mathcal D$ is 
	a maximally non--integrable distribution of hyperplanes on $C$, that is, a smooth assignment at each point $p$ of $C$ of a hyperplane in 
	the tangent space $T_{p}C$.
	\end{definition}
One can prove that $\mathcal D$ can always be given locally as the kernel of a 1--form $\eta$ satisfying the condition $\eta\wedge (d\eta)^{n}\neq 0$,
where $\wedge$ is the wedge product and  $(d\eta)^{n}$ means $n$ times the wedge product
of the 2--form $d\eta$. This characterization will be enough for our purposes, and
therefore we can introduce the following less general but more direct definition.
	\begin{definition}\label{def:excontactM}\rm
	An \emph{exact contact manifold} is a pair $(C,\eta)$, where $C$ is a ($2n+1$)--dimensional manifold and $\eta$ is a 1--form
	satisfying $\eta\wedge (d\eta)^{n}\neq 0$.
	\end{definition}
In what follows we will always restrict to the case of exact contact manifolds.

As it is standard in geometry, transformations that preserve the contact structure, and hence the contact geometry,
play a special role on these spaces.  
By noticing that the geometric object of interest is the kernel of a 1--form, one then defines isomorphisms in the contact 
setting in the following way.

\begin{definition}\label{def:contacttransform}\rm
A \emph{contact transformation} or \emph{contactomorphism} $F:(C_{1},\eta_{1})\rightarrow (C_{2},\eta_{2})$
is a map that preserves the contact structure
    \begin{equation}
    F^{*}\eta_{2}=\alpha_{F}\eta_{1},
    \end{equation}
where $F^{*}$ is the pullback induced by $F$, and $\alpha_{F}:C_{1}\rightarrow \mathbb R$ 
is a nowhere--vanishing function.
\end{definition}

\begin{remark}\label{rem:contacttransformation}\rm

In words, Definition~\ref{def:contacttransform} 
states 
that a contact map re--scales the contact $1$--form by multiplying it by a 
nowhere--vanishing function. Indeed, such multiplication preserves the kernel of 
the $1$--form, and hence the resulting geometry.

\end{remark}

Let us present a simple but important example of contact manifold:
we take $C=\R^{2n+1}$ 
and specify the same contact structure in 2 different ways, corresponding to different choices of coordinates.
Afterwards we prove that it amounts to the same structure by providing an explicit contactomorphism between the two.

\begin{example}[The standard contact structure in canonical coordinates]\label{ex:std}\rm
The standard structure is defined as the kernel of the $1$--form
\begin{equation}\label{eq:standard1form}
  \eta_{\mathrm{std1}}:=dS-P dX\,.
\end{equation}
\end{example}
We use ``standard'' because one can show that a contact structure on any manifold looks 
like this one locally~\cite{geiges2008introduction}.
\begin{example}[The standard contact structure in non--canonical coordinates]\label{ex:H}\rm
This structure is defined as the kernel of the $1$--form
\begin{equation}\label{eq:H1form}
  \eta_{\mathrm{std2}}:=d\tilde S-\frac{1}{2}\tilde Pd\tilde X+\frac{1}{2}\tilde Xd \tilde P\,.
\end{equation}
\end{example}
Although this appears different from the structure in Example~\ref{ex:std} they define
equivalent geometries, as we now show.
\begin{remark}\label{rem:contact12}\rm
We can explicitly construct a contact transformation between  $\eta_{\mathrm{std1}}$
and $\eta_{\mathrm{std2}}$ above. The map
\begin{equation}\label{eq:contact12}
	F:(X,P,S)\mapsto\left(\tilde X={X+P},\tilde P=\frac{P-X}{2},\tilde S=S-\frac{XP}{2}\right)
\end{equation}
satisfies $F^{*}\eta_{\mathrm{std2}}=\eta_{\mathrm{std1}}$. Consequently the two 
structures defined in Examples~\ref{ex:std} and~\ref{ex:H} are equivalent.  The
superficial difference arises only because they are written in different coordinates.
\end{remark}

Historically, one of the main motivations to introduce contact geometry is the study of time--dependent Hamiltonian systems on a symplectic manifold.
We briefly sketch these ideas because they will be relevant for our discussion.
	\begin{definition}\label{def:symplM}\rm
	A \emph{symplectic manifold} is a pair $(M,\Omega)$ where $M$ is a $2n$--dimensional manifold and $\Omega$ a 2--form on $M$ that is 
	closed, $d\Omega=0$,
	and non--degenerate, $\Omega^{n}\neq 0$.
	\end{definition}
	\begin{definition}\label{def:SHam}\rm
	A \emph{Hamiltonian system on a symplectic manifold} is a vector field $X_{H}$ which is defined by the condition 
		\begin{equation}\label{eq:HamS}
		\Omega(X_{H},\cdot)=-dH
		\end{equation}
	where $H:M\rightarrow \mathbb R$ is a function called the \emph{Hamiltonian}. The equations for the integral curves of $X_{H}$
	are usually called \emph{Hamilton's equations}, and they are the foundations of the geometric treatment of mechanical systems.
	\end{definition}

Analogously to what happens in the contact case, 
it can be proved that all symplectic manifolds locally look the same; this result is known as Darboux's Theorem~\cite{abraham1978foundations}. 
In particular, they all look like the Euclidean space  $\R^{2n}$ with 
coordinates $(X,P)$, where $X\in\R^{n}$ play the role in physics 
of the \emph{generalized coordinates}, and $P\in\R^{n}$ of the corresponding \emph{momenta}.
In such coordinates  $\Omega=dP\wedge dQ$, and thus Hamilton's equations~\eqref{eq:HamS} read like~\eqref{eq:H1} and~\eqref{eq:H2}, but with the important
difference that $H:M\rightarrow \mathbb R$ does not depend on time, thus giving rise to the so--called \emph{conservative mechanical systems}
(the name is due to the fact that $H(X,P)$ is usually the total mechanical energy of the system and it is easy to show that it is conserved along the
dynamics).

To allow for $H$ to depend also on time and thus to describe \emph{dissipative systems} in the geometric setting, one usually performs the following extension:
first extend the manifold $M$ to $M\times\mathbb R$, where time $t$ is defined to be the coordinate on $\mathbb R$.
At this point $H(X,P,t)$ is a well defined function on $M\times\mathbb R$. Now, define (locally at least) the 1--form on $\theta=PdX-H(X,P,t)dt$. This is called
the \emph{Poincar\'e--Cartan} 1--form.
Finally, define the dynamics $\tilde X_{H}$ by the two conditions 
	\begin{equation}\label{eq:PCH}
	d\theta(\tilde X_{H},\cdot)=0, \qquad \theta(\tilde X_{H})=1\,.
	\end{equation}
It follows that the resulting equations for the integral curves of $\tilde X_{H}$ in the coordinates $(X,P,t)$ are just~\eqref{eq:H1} and~\eqref{eq:H2}
for a generic $H$ (now with the explicit
time dependence on $H$), plus a trivial equation for $t$, namely $\dot t=1$.

	\begin{remark}\label{rem:PC}\rm
	For us there are four important points that are worth being remarked about this construction:
	\begin{enumerate}
	\item $\theta$ is a contact form on $M\times\mathbb R$ (unless $H$ is a homogeneous function of degree 1 in $P$, a case which is not interesting here).
	\item From the point of view of contact geometry, 
	the dynamics in \eqref{eq:PCH} is a very special type of Hamiltonian system on a contact manifold (to be defined below), corresponding to 
	a constant Hamiltonian function with value -1 (see Remark~\ref{rem:inclusion} below).
	\item Notice that, although Hamilton's equations~\eqref{eq:H1} and~\eqref{eq:H2} 
	dissipate the energy function $H(X,P,t)$, from the geometric point of view the system
	of Hamilton's equations~\eqref{eq:H1} and~\eqref{eq:H2} and $\dot t=1$ is conservative, 
	because $\pounds_{\tilde X_{H}}\theta=0$, 
	where $\pounds_{\tilde X_{H}}\theta$ denotes the Lie derivative of $\theta$ with respect to 
    $\tilde X_{H}$.
	This means that the dynamics preserves $\theta$,
	and hence the volume of the space given by $\theta\wedge (d\theta)^{n}$.
	\item The \emph{pre--symplectic dynamics} used in~\cite{francca2020dissipative} is exactly the dynamics of $\tilde X_{H}$. 
	Indeed one can check
	that their final manifold of states and dynamical equations, after fixing the appropriate gauge, 
	coincide with $M\times \mathbb R$ and~\eqref{eq:PCH} respectively
	(indeed, it suffices to specialize our discussion to the case $M=T^{*}\mathcal M$, and note that $d\theta=\Omega$ in their notation).
	Similarly, also the \emph{cosymplectic dynamics} used in~\cite{campos2021discrete} is precisely the dynamics of $\tilde X_{H}$.
	\end{enumerate}
	\end{remark}

We can now define dynamical systems that generalize the Hamiltonian systems arising
in symplectic geometries.

\begin{definition}[Contact Hamiltonian systems]\label{def:CHS}\rm
Given a possibly time--dependent differentiable function $\cH$ on a contact manifold 
$(C,\eta)$, we define the \emph{contact Hamiltonian vector 
field associated to $\cH$}  as the vector field $X_{\cH}$ satisfying
\begin{equation}\label{eq:CHVF}
  \pounds_{X_{\cH}}\eta=-R(\cH)
  \,\eta \qquad \eta(X_{\cH})=-\cH\,,
\end{equation}
where $\pounds_{X_{\cH}}\eta$ denotes the Lie derivative of $\eta$ with respect to 
$X_{\cH}$ and $R$ is the \emph{Reeb vector field}, that is, the vector field satisfying
$d\eta(R,\cdot)=0$ and $\eta(R)=1$.
We denote the flow of $X_{\cH}$ the \emph{contact Hamiltonian system 
associated to~$\cH$}.
\end{definition}
	
\begin{remark}\rm
The first condition in~\eqref{eq:CHVF} simply ensures that the flow of $X_{\cH}$
generates contact transformations, while the second condition requires the vector 
field to be generated by a Hamiltonian function.
\end{remark}
\begin{remark}\label{rem:inclusion}\rm
	It follows directly from Definition~\ref{def:CHS} and from identifying $C=M\times\mathbb R$ and $\eta=\theta$
	that the dynamics $\tilde X_{H}$ describing time--dependent symplectic Hamiltonian 
	systems is the contact Hamiltonian system corresponding to the contact Hamiltonian
	$\cH=-1$, and as such it is a very particular instance of a contact Hamiltonian system.
	In our discussion we will consider more general instances, in the sense that the contact manifold $C$ is not restricted to 
	be of the form $M\times \mathbb R$ and the Hamiltonian can be any function on $C$.
	Notice that for a general contact Hamiltonian system it follows directly from~\eqref{eq:CHVF}
	that the derivative along the flow of the Hamiltonian is $\pounds_{X_{\cH}}\cH=-R(\cH)\cH$, meaning that $\cH$ is not conserved. 
	For instance, when $R(\cH)$ is a strictly positive function,
	as we shall always consider below, 
	the function $\cH$ is constantly dissipated; however, more general dissipative terms may be allowed.
	A similarly direct calculation shows that $\pounds_{X_{\cH}}\eta\wedge(d\eta)^{n}=-(n+1)R(\cH)\eta\wedge(d\eta)^{n}$, 
	meaning that also the volume is contracted,
	to be compared with $\pounds_{\tilde X_{H}}\theta\wedge(d\theta)^{n}=0$, 
	as it happens for the geometric description of time--dependent symplectic Hamiltonian systems
	in the Poincar\'e--Cartan (or pre--symplectic) setting.
	Therefore we see that in the general contact case we have also a geometric description of dissipation, 
	defined as a contraction of the relevant volume form.
\end{remark}

As we did above where we introduced the standard model of contact manifolds in two different useful coordinate systems,
we write here the corresponding models for contact Hamiltonian systems.

\begin{lemma}[Contact Hamiltonian systems: std1]\label{def:CHSstandard}\rm
Given a (possibly time--dependent) differentiable function $\cH(X,P,S,t)$ on the 
contact state space $(\mathbb R^{2n+1},\eta_{\mathrm{std1}})$, the associated contact
Hamiltonian system is the following dynamical system
\begin{align}
\dot X &=
\nabla_{P}\cH\label{eq:CH1}
\\
\dot P &=-\nabla_{X}\cH-P\frac{\partial \cH}{\partial S}\label{eq:CH2}
\\
\dot S &=\langle\nabla_{P}\cH,P\rangle-\cH\label{eq:CH3}\,.
\end{align} 
\end{lemma}

\begin{lemma}[Contact Hamiltonian system: std2]\label{def:CHSHeisenberg}\rm
Given a (possibly time--dependent) differentiable function $\cH(X,P,S,t)$ on the 
contact state space $(\mathbb R^{2n+1}, \eta_{\mathrm{std2}})$, the associated
contact Hamiltonian system is the following dynamical system
\begin{align}
\dot X &=
\nabla_{P}\cH-\frac{1}{2}X\frac{\partial \cH}{\partial S}\label{eq:CHH1}
\\
\dot P &=-\nabla_{X}\cH-\frac{1}{2}P\frac{\partial \cH}{\partial S}\label{eq:CHH2}
\\
\dot S &=\frac{1}{2}\left(\langle X,\nabla_{X}\cH\rangle+\langle P,\nabla_{P}\cH\rangle\right)-\cH\label{eq:CHH3}\,.
\end{align} 
\end{lemma}
The proofs of the above lemmas follow from writing explicitly the conditions in~\eqref{eq:CHVF}
for $\eta_{\mathrm{std1}}$ and $\eta_{\mathrm{std2}}$ respectively, using Cartan's identity 
for the Lie derivative of a 1--form, and {from the fact that 
$R=\partial/\partial S$ in both cases, as can be seen by writing its definition in the corresponding coordinates.}
	
\begin{remark}[Variational formulation]\label{rem:variational}\rm
Contact systems can alternatively be introduced starting from Herglotz' variational principle~\cite{georgieva2003generalized,vermeeren2019contact,de2019singular,anahory2021geometry}, with 
the Lagrangian function 
$\cL(X,V,S,t)$ and its corresponding generalized Euler--Lagrange equations
\begin{equation}\label{eq:GEL}
\frac{d}{dt}\left(\frac{\partial \cL}{\partial V}\right)-\frac{\partial \cL}{\partial X}-\frac{\partial \cL}{\partial V}\frac{\partial \cL}{\partial S}=0\,,
\end{equation}
together with the action equation
\begin{equation}\label{eq:action}
\dot S=\cL(X,V,S,t)\,.
\end{equation}
Indeed, for regular Lagrangians it can be shown that~\eqref{eq:CH1}--\eqref{eq:CH3} are 
equivalent to the system~\eqref{eq:GEL}--\eqref{eq:action}. 
\end{remark}

\subsection{Main results}\label{subsec:main}
We arrive at our first main result with a direct calculation using equations~\eqref{eq:CH1}--\eqref{eq:CH3} and~\eqref{eq:CHH1}--\eqref{eq:CHH3}, 
\begin{proposition}[Recovering previous frameworks]\label{prop:recover}\rm
All the previously--mentioned frameworks for describing continuous--time optimization methods 
can be recovered as follows:
\begin{itemize}

\item[i)] 
If $\cH=H(X,P,t)$, that is, if $\cH$ does not depend explicitly on $S$, then from 
equations~\eqref{eq:CH1}--\eqref{eq:CH2} we obtain the standard Hamiltonian 
equations~\eqref{eq:H1}--\eqref{eq:H2}, with~\eqref{eq:CH3} completely decoupled from the system.  
In particular, this includes the Bregman dynamics~\eqref{eq:H1} and \eqref{eq:H2} as a special case of contact Hamiltonian systems.
 	
\item[ii)] 
If $\cH=H(X,P)+c\,S$, then from equations~\eqref{eq:CH1}--\eqref{eq:CH2} we obtain the 
standard equations for conformally symplectic systems~\eqref{eq:CS1}--\eqref{eq:CS2}, 
with~\eqref{eq:CH3} once again completely decoupled from the system.  


\item[iii)] 
If $\cH=H(X,P)+\langle X^{*},P\rangle-\langle P^{*},X\rangle+2S$, then from equations~\eqref{eq:CHH1}--\eqref{eq:CHH2} we 
obtain the Hamiltonian descent equations~\eqref{eq:HD1}--\eqref{eq:HD2}, with~\eqref{eq:CHH3} 
decoupled from the system.  

\item[iv)]
If $\cH=\frac{1}{2}\norm{P}^{2}+f(X)+\frac{3}{t}S$, then from equations~\eqref{eq:CH1}--\eqref{eq:CH2} we 
obtain the continuous limit of NAG~\eqref{eq:DiffEqNAG}, with~\eqref{eq:CH3} 
decoupled from the system. 

\end{itemize}

\end{proposition}

Three immediate but powerful consequences of Proposition~\ref{prop:recover}
are given in the following corollaries.
	\begin{corollary}\label{cor:convergencerates}\rm
All the convergence analyses of the above dynamics provided in the corresponding literature hold.
	\end{corollary}
In particular, by choosing the contact Hamiltonian $\cH=H_{\ttt{Br}}(X,P,t)$, and for the appropriate choice
of the free functions $\alpha, \beta$ and $\gamma$  we can obtain contact systems with polynomial convergence rates of any order 
and even exponential convergence rates (see Equation~\ref{eq:ExpConv}). 
So we see that in the class of contact Hamiltonian systems
we can at least reproduce all the conventional approaches to convergence rates from the Hamiltonian perspective.
Moreover, contact systems provide the opportunity to generalize all these dynamics, for instance by fixing nonlinear dependences
on the additional variable $S$, a strategy that will be considered in future works.
As a second consequence of Proposition~\ref{prop:recover}, we have
    \begin{corollary}\rm
    All the dynamics in Proposition~\ref{prop:recover} have a variational formulation.
    \end{corollary}
This important observation outlines the fact that all such dynamics may be studied by variational methods,
either in the continuous case (see e.g.~the interesting analysis in~\cite{zhang2021rethinking} and compare with~\cite{ryan2022action}) 
or in the discrete one (cf.~\cite{campos2021discrete} and~\cite{vermeeren2019contact,anahory2021geometry}).
All these aspects can now be analyzed through the lens of the corresponding contact tools, cf.~Remark~\ref{rem:variational}.
In this work we focus on the following (third) key aspect of the geometric approach:
	\begin{corollary}\label{cor:invariance}\rm
All the systems in Proposition~\ref{prop:recover} 
share invariance under a very large group of transformations, the group 
of (time--dependent) contact transformations.
	\end{corollary}
Corollary~\ref{cor:invariance} has interesting practical ramifications:
on the one hand, invariance under all contact transformations guarantees that all such dynamics, and the corresponding algorithms
obtained through geometric discretizations, are much less sensitive to the conditioning of the problem
rather than other algorithms that do not possess a geometric structure 
(we refer to~\cite{maddison2018hamiltonian} for an illuminating discussion on this, 
and we remark that contact transformations are more general than the symplectic transformations considered therein in Section~2.1).
	On the other hand,	one can exploit contact
	transformations to re--write the dynamics in particular sets of 
	coordinates where the system takes a simpler form, with great benefit for the ensuing discretization.
	
As a first ``warm-up'' example of the utility of contact transformations in optimization, we now prove that
NAG is a composition of a contact transformation followed by a simple gradient descent step.
 This result is inspired by the conjecture put forward in~\cite{betancourt2018symplectic},
who argued that \emph{symplectic} maps followed by gradient descent steps can generate the exponential convergence 
near convex optima empirically observed in discrete--time NAG.
Here instead we provide an actual proof that NAG is based on the composition of a \emph{contact} map and
a gradient step.
	 
\begin{proposition}[NAG is contact $+$ gradient]\rm	 
Discrete--time NAG,~\eqref{eq:NAG1}--\eqref{eq:NAG2}, is 
{given by the composition of a contact
map and a gradient descent step.}
\end{proposition}
\begin{proof}
First we recall from Definition~\ref{def:contacttransform} 
that a contact transformation for the contact structure given by~\eqref{eq:H1form} is a map that satisfies
	\begin{equation}\label{eq:contactmap}
	dS_{k+1}-\frac{1}{2}P_{k+1}dX_{k+1}+\frac{1}{2}X_{k+1}dP_{k+1}=f(X_{k},P_{k},S_{k})\left(dS_{k}-\frac{1}{2}P_{k}dX_{k}+\frac{1}{2}X_{k}dP_{k}\right)\,,
	\end{equation}
for some function $f(X_{k},P_{k},S_{k})$ that is nowhere $0$.
Then one can directly verify that NAG can be exactly decomposed in the contact state space as the composition of the map
	\begin{align}
	X_{k+1}&=P_k\\
	P_{k+1}&=X_{k+1}+\frac{k-1}{k+2} (X_{k+1}-X_{k})\,,\label{eq:C1b}\\
	S_{k+1}&=\frac{k-1}{k+2}S_{k}\,,\label{eq:Sbis}
	\end{align}
which is readily seen to be a contact transformation satisfying
	\begin{align}
	dS_{k+1}-\frac{1}{2}P_{k+1}dX_{k+1}+\frac{1}{2}X_{k+1}dP_{k+1}
				&=\frac{k-1}{k+2}\left[dS_{k}-\frac{1}{2}P_{k}dX_{k}+\frac{1}{2}X_{k}dP_{k}\right]\,,\label{eq:contactNAG}
	\end{align}
 followed 
by a standard gradient descent map, 
	\begin{align}
	X_{k+1}&=X_{k}-\tau\nabla f(X_{k})\label{eq:C1a}\\
	P_{k+1}&=P_{k}\,\label{eq:C2b}\\
	S_{k+1}&=S_{k}\,.
	\end{align}
\end{proof}

Now we come to the main point of our work: 
in order to further illustrate the utility of contact transformations
in a case of great current interest in the literature
(see e.g.~\cite{sun2020continuous,wilson2021lyapunov,francca2021optimization,duruisseaux2022variational,duruisseaux2022time}),
in the remainder of this section we focus on the Bregman dynamis and show how to use time--dependent contact
transformation in order to re--write the Bregman dynamics in its most general form in such a way that it is clear
that it is \emph{always} derived from a separable Hamiltonian, 
and is thus amenable of simple, geometric and explicit discretizations by direct splitting. 
	
	As a first step, we need to introduce 
	time--dependent
	contact transformations, which we do now following closely~\cite{bravetti2017contact},
and proceeding analogously to the case of time--dependent canonical (symplectic) 
transformations usually encountered in classical mechanics~\cite{Arnold}. 
First, we extend the contact manifold $C$ to $C\times \mathbb R$  by including the time variable as a coordinate on $\mathbb R$;
then we also extend the contact
form $\eta$ to
	\begin{equation}\label{etaE}
	\eta^{E}=\eta+\cH(X,P,S,t)dt\,,
	\end{equation}
where $\cH(X,P,S,t)$ is the contact Hamiltonian of the system.
Recall as usual that locally we can think $C\simeq\mathbb R^{2n+1}$ and
	$\eta^{E}=dS-PdX+\cH(X,P,S,t)dt$.
Then, we define a time--dependent contact transformation as follows.
	\begin{definition}[Time--dependent contact transformation]\label{def:tcontacttr} \rm
	A \emph{time--dependent contact transformation} for a system with contact Hamiltonian $\cH$ is a map
	$F:C\times \mathbb R\rightarrow C\times \mathbb R$
	such that 
	\begin{equation}\label{eq:TDCT1}
	F^{*}\eta^{E}=\sigma\,\eta^{E} \qquad \text{and} \qquad F^{*}t=t\,,
	\end{equation}
	where $\sigma:C\times \mathbb R\rightarrow  \mathbb R$ is a nowhere--vanishing function.
	\end{definition}
	\begin{remark}\rm
	In local 
	coordinates $(X,P,S,t)$ on $C\times \mathbb R$, this is equivalent to saying that a contact transformation maps to new coordinates
	$$\tilde X=\tilde X(X,P,S,t),\quad \tilde P=\tilde P(X,P,S,t),\quad \tilde S=\tilde S(X,P,S,t),\quad \tilde t=t\,,$$
	such that
	\begin{equation}\label{tcontact}
	d\tilde S-\tilde Pd\tilde X+K(\tilde X,\tilde P,\tilde S,t)dt= \sigma(X,P,S,t)\left(dS-PdX+\cH(X,P,S,t)dt\right)\,,
	\end{equation}
	where $K(\tilde X,\tilde P,\tilde S,t)$ is the new contact Hamiltonian in the new coordinates.
	\end{remark}
It follows directly from~\eqref{tcontact}, that a necessary condition for a time--dependent transformation to be contact is
	\begin{equation}\label{HKcondition}
	\frac{\partial \tilde S}{\partial t}-\tilde P\frac{\partial\tilde X}{\partial t}+K=\sigma\,\cH\,,
	\end{equation}
that is, the new contact Hamiltonian is found from the original one according to 
	\begin{equation}\label{HKcondition2}
	K=\sigma\,\cH-\frac{\partial \tilde S}{\partial t}+\tilde P\frac{\partial\tilde X}{\partial t}\,.
	\end{equation}
Finally, it can be proved that the time--dependent contact Hamiltonian equations are preserved under the above transformations 
if and only if 
	\begin{equation}\label{eq:TDCTcondition}
	-\sigma\frac{\partial K}{\partial \tilde S}=-\sigma \frac{\partial \cH}{\partial S}-d\sigma(X_{\cH}^{E})\,,
	\end{equation}
where $X_{\cH}^{E}=X_{\cH}+\partial_{t}$	
	(see~\cite{bravetti2017contact} for more details).

Now we can use the above results to simplify the Bregman Hamiltonian. In particular, we construct
a time--dependent contact transformation that brings $H_{\ttt{Br}}$ to a Hamiltonian 
$K_{\ttt{Br}}$ which is equivalent to the original one but separable, and hence directly amenable to 
an explicit discretization by splitting.
Let us consider the transformation
	\begin{equation}\label{ctr1}
	\tilde X=X,\quad \tilde P=e^{-\gamma}p+\nabla h(X),\quad \tilde S=e^{-\gamma} S+h(X),\quad \tilde t=t\,,
	\end{equation}
with $h(X)$ the same convex function as in~\eqref{eq:BH}.
We have the following
	\begin{lemma}\rm
	For any contact Hamiltonian system with Hamiltonian $\cH$,
	the transformation~\eqref{ctr1} is not canonical (symplectic) but it is contact, with 
	\begin{equation}\label{sigmaK}
	\sigma(X,P,S,t)=e^{-\gamma},\qquad
	K=e^{-\gamma}\left(\cH+\dot\gamma S\right)\,.
	\end{equation}
	\end{lemma}
	\begin{proof}
	The proof proceeds by directly calculating the differential of the transformation~\eqref{ctr1}.
	In this way we obtain that 
	$$d\tilde P\wedge d\tilde X=e^{-\gamma}dP\wedge dX\,,$$
	showing that it is not a canonical (symplectic) transformation. Moreover, one can also verify that
	$$d\tilde S-\tilde Pd\tilde X+K(\tilde X,\tilde P,\tilde S,t)dt= e^{-\gamma}\left(dS-PdX+\cH(X,P,S,t)dt\right)\,,$$
	and that~\eqref{eq:TDCTcondition} holds for $K$ as in~\eqref{sigmaK}, showing that indeed it is a time--dependent contact transformation 
	that preserves the dynamics.
	\end{proof}

Now we are finally ready to prove the second main result of our work.
	\begin{theorem}[The Bregman Hamiltonian is always separable]\label{th:Kbr}\rm
	Under the ideal scaling condition $\dot\gamma=e^{\alpha}$, 
	the generally non--separable Bregman Hamiltonian
	\begin{equation}\label{eq:BHbelow}
	H_{\ttt{Br}}(X,P,t)=\e^{\alpha+\gamma}\left(D_{h^*}(\e^{-\gamma}P+\nabla h(X),\nabla h(X))+\e^{\beta}f(X)\right)\,,
	\end{equation}
	is equivalent to the separable contact Hamiltonian
	\begin{equation}\label{eq:BK3}
	K_{\ttt{Br}}(\tilde X,\tilde P,\tilde S,t)=\e^{\alpha}\left(h^{*}(\tilde P)-\langle\tilde P,\tilde X\rangle+\e^{\beta}f(\tilde X)+\tilde S\right)\,,
	\end{equation}
	for any choice of the convex function $h(X)$.
	\end{theorem}
	\begin{proof}
We start by applying the above contact transformation~\eqref{ctr1} 
to the Bregman Hamiltonian~\eqref{eq:BHbelow} and use that $\dot\gamma=e^{\alpha}$ (ideal scaling condition from~\cite{wibisono2016variational})
to obtain the new Bregman Hamiltonian in the new coordinates, which, using~\eqref{sigmaK}, reads
	\begin{equation}\label{eq:BK}
	K_{\ttt{Br}}(\tilde X,\tilde P,\tilde S,t)=\e^{\alpha}\left(D_{h^*}(\tilde P,\nabla h(\tilde X))+\e^{\beta}f(\tilde X)+\tilde S-h(\tilde X)\right)\,.
	\end{equation}
Interestingly, by the very definition of the Bregman divergence, Equation~\eqref{eq:BD}, we can rewrite the above expression as
	\begin{equation}\label{eq:BK2}
	K_{\ttt{Br}}(\tilde X,\tilde P,\tilde S,t)=\e^{\alpha}\left(h^{*}(\tilde P)-h^{*}(\nabla h(\tilde X))-\langle\tilde P,\tilde X\rangle+\langle\nabla h(\tilde X),\tilde X\rangle+\e^{\beta}f(\tilde X)+\tilde S-h(\tilde X)\right)\,.
	\end{equation}
Moreover, by definition of the Legendre transform, 
	$$h^{*}(\nabla h(\tilde X))=\langle\nabla h(\tilde X),\tilde X\rangle-h(\tilde X)\,,$$ 
and thus some of the terms in~\eqref{eq:BK2} cancel, to give
the sought--for final result~\eqref{eq:BK3}.
	\end{proof}
	

We remark at this point that non--separable Hamiltonians systems are notoriously hard to discretize, 
since a direct application of standard symplectic (or contact)
integrators leads to implicit algorithms. 
To bypass this problem for the Bregman dynamics, in~\cite{francca2020dissipative}
the authors suggested to use a technique first proposed in~\cite{tao2016explicit} that consists in doubling 
the phase space and then propose an ``augmented Hamiltonian'' on such space that is separable and thus can be integrated using standard explicit
symplectic integrators. However, it must be remarked that the geometric (pre--symplectic) 
character of such algorithms is only guaranteed in the doubled phase
space, not in the original one, and this signals that such procedure has to be handled with care.
Instead, using Theorem~\ref{th:Kbr}, we will be able to perform a
direct splitting that on the one side simplifies the 
algorithm and requires less computational burden, and on the other side guarantees structure--preservation
in the original phase space of the system.

\subsection{Relativistic Bregman algorithm}\label{subsec:RBA}
To illustrate the benefits of our approach and motivated by the arguments
in~\cite{francca2019conformal} on the advantage of choosing a ``relativistic'' kinetic term in the Hamiltonian
approach to optimization, 
we first define the \emph{Relativistic Bregman dynamics} as a motivating example of
a non--separable Bregman Hamiltonian that can be turned into an equivalent (but separable) one by means of
Theorem~\ref{th:Kbr}.
Afterwards, by a direct splitting we obtain the corresponding \emph{Relativistic Bregman algorithm} (RB).

Let us consider the case in which $H_{\ttt{Br}}$ in~\eqref{eq:BH} is generated by the convex function
	\begin{equation}\label{eq:relativistich}
	h(X)=-mv^{2}\sqrt{1-\frac{\norm{X}^{2}}{v^{2}}}\,,
	\end{equation}
where this function is inspired in the relativistic Lagrangian for a particle of mass $m$ and 
$v$ is the analogue of the speed of light. 
Moreover, let us fix the functions 
    $$\alpha=\log c-\log t, \qquad \beta=c\log t +\log C, \qquad  \gamma=c\log t, \qquad c,C>0\,, $$ 
which satisfy 
the ideal scaling conditions~\eqref{eq:IS1} and \eqref{eq:IS2}.
With this choice, the Bregman dynamics is guaranteed to have a polynomial convergence rate of order $c$,~\cite{wibisono2016variational}.

From~\eqref{eq:relativistich} and~\eqref{eq:BH} we obtain the \emph{Relativistic Bregman Hamiltonian}
	\begin{equation}
	H_{\ttt{Br}}^{R}={\small\e^{\alpha+\gamma}\left[v\sqrt{\norm{\e^{-\gamma}P+\frac{mX}{\sqrt{1-\frac{\norm{X}^{2}}{v^{2}}}}}^{2}+m^{2}v^{2}}
			-\frac{mv^{3}}{\sqrt{v^{2}-\norm{X}^{2}}}-\e^{-\gamma}\langle P,X\rangle+\e^{\beta}f(X)\right]},\label{eq:HBrRel}
	\end{equation}
which is clearly non--separable.
However, using~\eqref{eq:BK3} we get the much simpler but equivalent
	\begin{equation}\label{eq:KBrRel}
	K_{\ttt{Br}}^{R}=\e^{\alpha}\left(v\sqrt{\norm{P}^{2}+m^{2}v^{2}}-\langle  P, X\rangle+\e^{\beta}f(X)+S\right)\,,
	\end{equation}
where, from now on, we drop the notation with a 
tilde above the new coordinates in~\eqref{eq:KBrRel}, as there will be no problem of confusion.
We emphasize that the Relativistic Bregman Hamiltonian~\eqref{eq:HBrRel} is different from the Relativistic
Gradient Descent proposed in~\cite{francca2019conformal}. Indeed, although both are defined using a relativistic kinetic term, the
former belongs to the Bregman family of Hamiltonian functions, while the latter does not.

In~\cite{francca2020dissipative} in order to integrate Hamilton's equation stemming from a non--separable Hamiltonian like~\eqref{eq:HBrRel},
the authors suggest to use the technique first proposed in~\cite{tao2016explicit}, 
according to which one first doubles the phase space dimensions to a space 
with coordinates $(X,P,t,\bar X,\bar P,\bar t)$, and then defines the \emph{augmented Hamiltonian}
	\begin{equation}\label{eq:AH}
	H(X,P,t,\bar X,\bar P,\bar t)=H_{\ttt{Br}}^R(X,\bar P,t)+H_{\ttt{Br}}^R(\bar X, P,\bar t)+\frac{\xi}{2}\left(\norm{X-\bar X}^{2}+
	\norm{P-\bar P}^{2}\right)\,,
	\end{equation}
where $\xi>0$ is a free parameter whose value controls the strength of the last term, that is included in order to bias the dynamics towards
$X=\bar X$ and $P=\bar P$, and whose value has to be tuned in practice. 
The equations of motion for the augmented Hamiltonian are equivalent to those of the original system
when $X=\bar X,P=\bar P$, and $t=\bar t$ and the Hamiltonian is now separable and thus one can integrate the dynamics using a splitting scheme.

On the other hand, the Hamiltonian $K_{\ttt{Br}}^{R}$ in~\eqref{eq:KBrRel} 
can be splitted directly in the original contact phase space, with coordinates
$(X,P,S)$ (and time). For instance, using the splitting
	\begin{equation}
K_{A}:=e^{\alpha} v \sqrt{\norm{P}^{2}+v^{2} m^{2}}, \quad K_{B}=-e^{\alpha} \langle X, P\rangle, \quad K_{C}=e^{\alpha+\beta} f(X),\quad K_{D}=e^{\alpha}S
	\end{equation}
we get the maps
$$
\begin{aligned}
&\varphi_{\tau}^{A}\left(\begin{array}{c}
X \\
P \\
S
\end{array}\right)=\left(\begin{array}{c}
X+\frac{e^{\alpha} v\, P}{\sqrt{\norm{P}^{2}+v^{2} m^{2}}} \tau \\
P \\
S+e^{\alpha} v\left(\frac{-v^{2}m^{2}}{\sqrt{\norm{P}^{2}+v^{2} m^{2}}}\right) \tau
\end{array}\right) \\
&\varphi_{\tau}^{B}\left(\begin{array}{c}
X \\
P \\
S
\end{array}\right)=\left(\begin{array}{c}
X \exp \left(-e^{\alpha} \tau\right) \\
P \exp \left(e^{\alpha} \tau\right) \\
S
\end{array}\right) \\
&\varphi_{\tau}^{C}\left(\begin{array}{c}
X \\
P \\
S
\end{array}\right)=
\left(\begin{array}{c}
X \\
-e^{\alpha+\beta} \nabla f(X) \tau+P \\
-e^{\alpha+\beta} f(X) \tau+S
\end{array}\right)\\
&\varphi_{\tau}^{D}\left(\begin{array}{c}
X \\
P \\
S
\end{array}\right)=\left(\begin{array}{c}
X \\
P \exp \left(-e^{\alpha} \tau\right) \\
S \exp \left(-e^{\alpha} \tau\right)
\end{array}\right)
\end{aligned}
$$
and, using the composition 
 $$\varphi_{\tau / 2}^{D} \circ \varphi_{\tau / 2}^{B}\circ \varphi_{\tau / 2}^{C} \circ \varphi_{\tau}^{A} 
 \circ \varphi_{\tau / 2}^{C} \circ \varphi_{\tau / 2}^{B} \circ \varphi_{\tau / 2}^{D}\,,$$
we obtain a second--order contact integrator 
(we refer to e.g.~\cite{bravetti2020numerical} for the study of contact integrators derived by splitting.
See also~\cite{goto2021fast} for a different but related approach to optimization using these types of algorithms).
Clearly this integrator
yields an optimization algorithm
that we shall call the \emph{Relativistic Bregman algorithm} (RB).

We remark on the most important difference between 
the strategy put forward in~\cite{francca2020dissipative}
as compared to ours: while the integrators proposed in~\cite{francca2020dissipative} for the ``non--separable'' case
are geometric 
(pre--symplectic)
only in the enlarged phase space, thus not guaranteeing that such property holds in the original phase space of the system, 
the RB that we have just described, as well as any other algorithm based on the use of contact splitting
integrators after Theorem~\ref{th:Kbr}, are structure--preserving in the original (contact) phase space. 
Therefore we expect the RB to be both more stable and efficient, with direct benefits for the optimization task, 
especially as the dimension of the problem increases.
The numerical simulations reported in~\cite{Daza-Torres2022_GitHub} provide evidence for these conclusions.

\section{Numerical experiments}\label{sec:numerics}
The purpose of this section is just illustrative: we aim to show that the 
Relativistic Bregman algorithm (RB) proposed above can effectively be used in benchmark
optimization and machine learning tasks. 
To show this point, we report here the performance of RB on some reference examples
and present a comparison
with its Euclidean counterpart, namely the \emph{Euclidean Bregman algorithm} (EB), and two 
standard optimization algorithms such as NAG and CM.
We refer to~\cite{Daza-Torres2022_GitHub} for the details of the EB and for a much larger comparison
with further test functions and algorithms.

In all the following examples we set $P_0=0$ and $S_0=0$ whenever such initial conditions are required.

\begin{example}[Quadratic function] \rm
Let us start with a simple quadratic function
\begin{equation}\label{eq:quad}
f(X)=\frac{1}{2}X^T AX,\quad X \in \mathbb{R}^{500},\quad \lambda(A)\sim \mathcal{U}(10^{-3},1),
\end{equation}
where $A \in \mathbb{R}^{500 \times 500}$ is a positive--definite random matrix with eigenvalues 
uniformly distributed over the range $[10^{-3},1].$ 
In Figure~\ref{fig:ex_quadratic}(a), we compare the performance of EB and RB on this problem using  
initial condition $X_0=(1,1,...,1)$, 
step size $\tau=10^{-4}$, speed of light $v=1000$, mass $m=10^{-3}$, and varying $c \in \{ 2, 4, 8\}$. 
For these parameters, both RB and EB exhibit similar rates of convergence,
which are close to the theoretical ones, see Table~\ref{tab:quadratic_convrates}. 

\begin{figure}
\subfigure[\bf{Quadratic function}]{\includegraphics[scale=0.35]{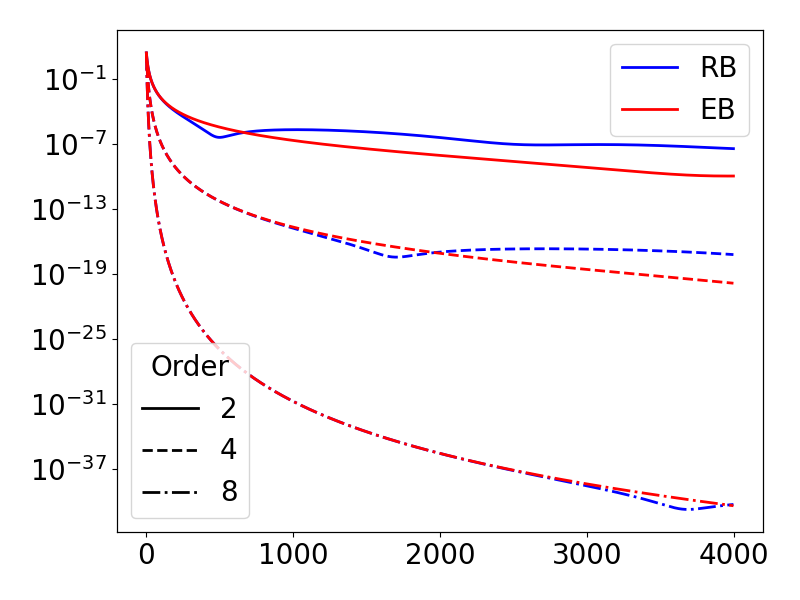}}
\subfigure[\bf{Quartic function}]{\includegraphics[scale=0.35]{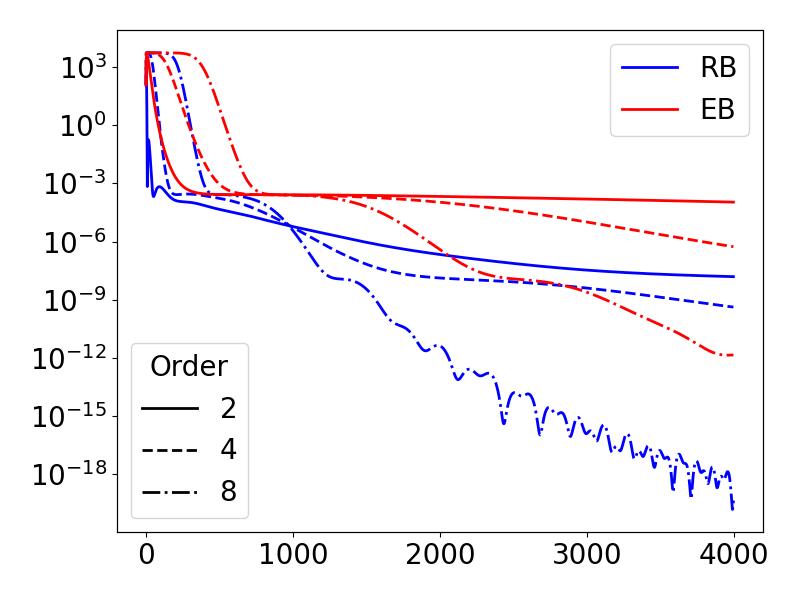}}
\caption{Comparison of the performance of the RB and EB methods, both with the same parameters, and for varying $c \in \{2,4,8\}.$}
\label{fig:ex_quadratic}
\end{figure}

\begin{table}[h!]
\centering
\renewcommand{\arraystretch}{1.5}
\begin{tabular}{|c|c c c |c c c |c c c|} \hline
\textbf{Order} & \multicolumn{3}{c|}{ $\mathbf{c=2}$} & \multicolumn{3}{c|}{$\mathbf{c=4}$} & \multicolumn{3}{c|}{$\mathbf{c=8}$}\\ \hline
Iterations & $10^3$ & $2\times10^3$ & $4\times10^3$ & $10^3$ & $2\times10^3$ & $4\times10^3$ & $10^3$ & $2\times10^3$ & $4\times10^3$ \\ \hline
RB  & $t^{-1.75}$ & $t^{-1.91}$ & $t^{-2.03}$ & $t^{-4.99}$ & $t^{-4.80}$ & $t^{-4.84}$ &$t^{-10.25}$ & $t^{-10.77}$ & $t^{-10.84}$\\
EB  & $t^{-2.26}$ & $t^{-2.59}$ & $t^{-2.37}$ & $t^{-4.89}$ & $t^{-5.19}$ & $t^{-5.67}$ &$t^{-10.25}$ & $t^{-10.76}$ & $t^{-11.21}$\\ \hline
\end{tabular}
\caption{Numerical convergence rates of the RB and EB algorithms for the Quadratic function~\eqref{eq:quad} achieved at 1000, 2000, and 4000 iterations
respectively.}
\label{tab:quadratic_convrates}
\end{table}
\end{example}

\begin{example}[Quartic function]\rm
Next let us consider the quartic function
\begin{equation}\label{eq:quart}
f(X)=[(X-1)^T \Sigma (X-1)]^2, \quad X \in \mathbb{R}^{50},\quad \Sigma_{ij}=0.9^{|i-j|}. 
\end{equation}
This convex function achieves its minimum value $0$ at $X^*=1$. 
In Figure~\ref{fig:ex_quadratic}(b), we compare the performance of EB and RB on this problem using 
initial condition $X_0 \sim \mathcal{U}(0,1)$, step size $\tau=10^{-3}$, and the rest of the parameters as in the previous example. 
In this case, RB shows better convergence rates than those of EB, being also closer to the theoretical ones, see Table~\ref{tab:quartic_convrates}.


\begin{table}[h!]
\centering
\renewcommand{\arraystretch}{1.5}
\begin{tabular}{|c|c c c |c c c |c c c|}\hline
\textbf{Order} & \multicolumn{3}{c|}{ $\mathbf{c=2}$} & \multicolumn{3}{c|}{$\mathbf{c=4}$} & \multicolumn{3}{c|}{$\mathbf{c=8}$}\\ \hline
Iterations & $10^3$ & $2\times10^3$ & $4\times10^3$ & $10^3$ & $2\times10^3$ & $4\times10^3$ & $10^3$ & $2\times10^3$ & $4\times10^3$ \\ \hline
RB  & $t^{-2.07}$ & $t^{-2.29}$ & $t^{-2.28}$ & $t^{-2.17}$ & $t^{-2.47}$ & $t^{-3.01}$ &$t^{-2.34}$ & $t^{-3.86}$ & $t^{-5.64}$\\
EB  & $t^{-1.20}$ & $t^{-1.16}$ & $t^{-1.24}$ & $t^{-1.21}$ & $t^{-1.35}$ & $t^{-2.03}$ &$t^{-1.22}$ & $t^{-2.30}$ & $t^{-3.69}$\\ \hline
\end{tabular}
\caption{Numerical convergence rates of the RB and EB algorithms for the Quartic function~\eqref{eq:quart}, achieved at 1000, 2000, and 4000 iterations
respectively.}
\label{tab:quartic_convrates}
\end{table}
\end{example}


\begin{example}[Machine Learning examples]\rm
Now we test the RB in two classical machine learning tasks,
and compare its performance with those of EB, CM and NAG. 
We use two popular datasets for classification from the UCI machine learning repository and 
try to fit a Two-regularized logistic regression model. 
The profiles of these datasets are summarized in Table \ref{tab:ML_data}. 
We set the regularization parameter for all methods to be $\lambda_{reg}=10^{-2}$; 
for the RB and EB algorithms, we set $v=1000$, $m=10^{-3}$, $C=1$, and $c=2$. 
The rest of the parameters were tuned in the validation dataset.

\begin{table}[!h]
\centering
\begin{tabular}{lcc} \hline
Dataset  &  Diabetes  &  Breast Cancer\\ \hline
Train       &  460  &  340 \\
Test        &  192  &  143\\
Validation  &  115  & 86 \\
Feat        &  8  &  6\\
Class       &  2  &  2 \\
\hline  
\end{tabular}
\caption{Profiles of the datasets.}
\label{tab:ML_data}
\end{table}

\subsubsection*{Diabetes}
In the first example, we use the Pima Indians Diabetes dataset. 
The objective is to predict, based on diagnostic measurements, whether a patient has diabetes. 
The dataset was separated into training 460, test 192, and validation 115 sets.  
Figure~\ref{fig:ml_diabetes} shows the mean over 50 random initial conditions, the 0.025 and 0.975 quantiles of the loss function, 
and the classification error evaluated in the test set. 
The RB algorithm converges faster to the optimum, which is reached by the other algorithms a few iterations later. 
The RB and EB algorithms showed less sensitivity to the initial condition compared to CM and NAG.
\begin{figure}[h!]
\centering
\subfigure[]{\includegraphics[scale=0.28]{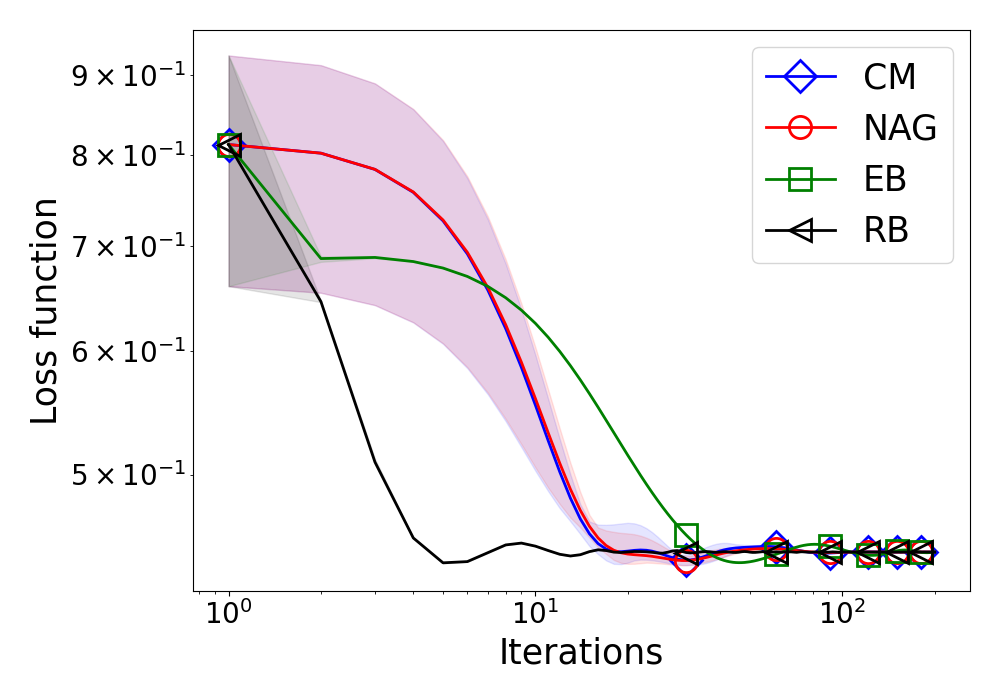}}
\subfigure[]{\includegraphics[scale=0.28]{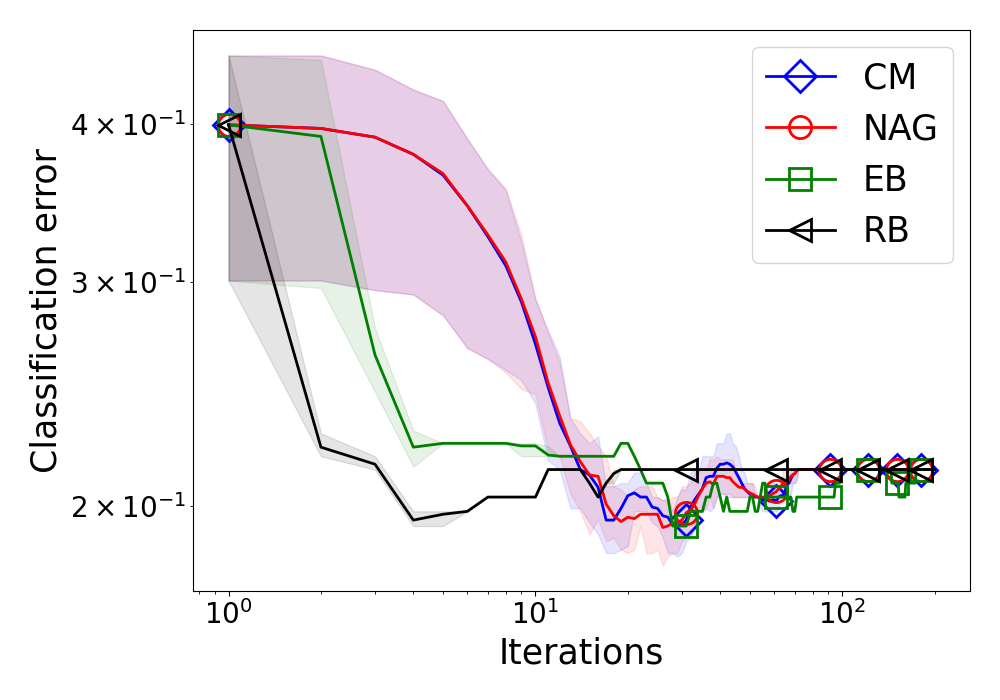}}
\caption{Two-regularized logistic regression model trained with RB, EB, CM, and NAG on the Pima Indians Diabetes dataset. 
(a) Objective function values and 
(b) classification error along with the iteration of optimization on log--log scale. 
The initial conditions are taken uniformly on $(0,1)$.} 
\label{fig:ml_diabetes}
\end{figure}

\subsubsection*{Breast Cancer}
The objective is to classify breast cancer from some features.  
The dataset was separated into training 460, test 192 and validation 115 sets.  
The RB algorithm exhibits similar behavior as in the previous example, converging faster than the other algorithms in the first few iterations, see Figure~\ref{fig:ml_cancer}.
\begin{figure}[h!]
\centering
\subfigure[]{\includegraphics[scale=0.28]{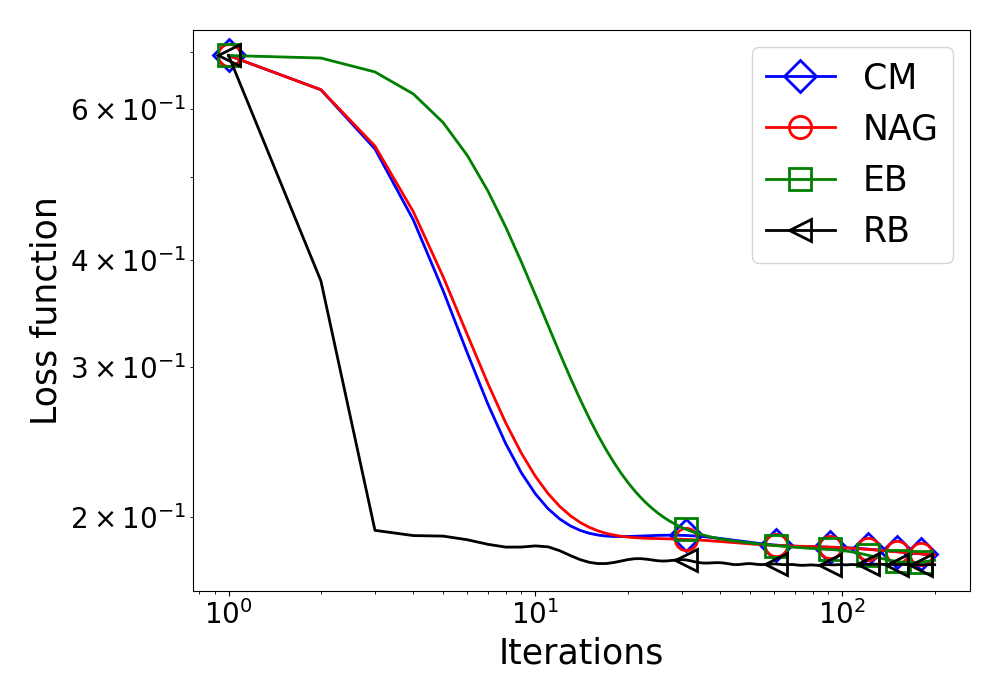}}
\subfigure[]{\includegraphics[scale=0.28]{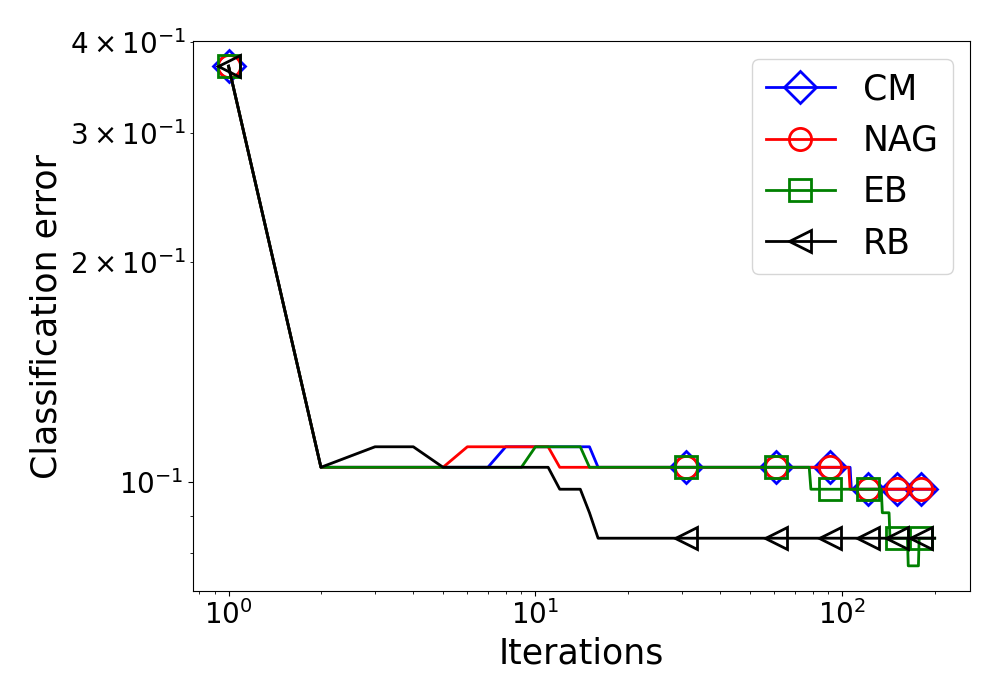}}
\caption{Two-regularized logistic regression model trained with RB, EB, CM, and NAG on the Breast Cancer dataset. 
(a) Objective function values and 
(b) classification error along with the iteration of optimization on a log--log scale. 
All the methods are initiated at the initial condition $X_0=0$.}
\label{fig:ml_cancer}
\end{figure}

\end{example}

\section{Conclusions}\label{sec:Conclusions}
Geometry is a powerful tool in pure and applied mathematics. 
Among other things, it enables to describe phenomena independently of the particular choice of coordinates.
This is because geometric objects are invariant under some group of transformations.
In this work, we have shown that
contact geometry and the related contact transformations can be extremely useful in optimization.
Indeed, we have shown that invariance under contact transformations is a common feature of all the recently--proposed dynamics
in the context of the dynamical systems approach to convex optimization (Proposition~\ref{prop:recover}). 
More importantly, we have proved that
the Bregman dynamics, which is perhaps the single most important recent finding in this context but whose implementation 
is seriously hindered by the fact that the Hamiltonian is non--separable, is actually always separable up to a contact transformation
(Theorem~\ref{th:Kbr}). This opens the way to applying explicit and fast numerical integration methods for simulating the Bregman dynamics, 
which in turn provides new efficient optimization algorithms, even when the original Hamiltonian looks non--separable.
Finally, in order to illustrate the relevance of considering algorithms that stem from non--separable Bregman Hamiltonians,
we have shown in some benchmark examples from the optimization and machine learning literature that the thus--proposed 
Relativistic Bregman algorithm compares favorably with
respect to the standard Classical Momentum and Nesterov's Accelerated Gradient algorithms, and also to the Euclidean Bregman algorithm.

In future work, we plan to perform a more in-depth analysis of the properties of the Relativistic Bregman algorithm.
Moreover, as commented in the text, it will be interesting to study several optimization dynamics through the lens of
Herglotz's variational principle. Finally, all the results presented here have a natural generalization to the case of
general differentiable manifolds,
and therefore we can extend our analysis with methods similar to those employed 
in~\cite{francca2021optimization,duruisseaux2022time,duruisseaux2022variational}.

\section*{Data availability statement}
The datasets analysed  and all the codes used for the study  are available at
\href{https://github.com/mdazatorres/Breg\_dynamic\_contact\_algorithm}{https://github.com/mdazatorres/Breg\_dynamic\_contact\_algorithm}.

%
%


\bibliographystyle{abbrvnat_mv}
\bibliography{contact.bib}


\end{document}